\documentclass{article}
\usepackage{amsmath,amssymb,indentfirst}
\usepackage[latin1]{inputenc}
\usepackage{url}
\newtheorem{thm}{Theorem}
\newtheorem{lem}{Lemma}

\newtheorem{rmk}{Remark}

\newcommand{\qed}{\hfill\mbox{\raggedright $\Box$}\medskip}

\begin{document}

\title{A Complex Primitive $N$th Root Of Unity: A Very Elementary Approach}

\author{Oswaldo Rio Branco de Oliveira}
\date{}
\maketitle

\begin{abstract}

This paper presents a primitive $n$th root of unity in $\mathbb C$. The approach is very elementary and avoids the following: the complex exponential function, trigonometry, and group theory. It also avoids differentiation, integration, and series. The presentation of the primitive is indirect. 
\end{abstract}

\vspace{0,2 cm}

\hspace{- 0,6 cm}{\sl Mathematics Subject Classification: 97F50, 97I80}

\hspace{- 0,6 cm}{\sl Key words and phrases:} Complex Numbers, Complex Analysis.

\vspace{0,3 cm}

\section{Introduction.}

This article presents a primitive $n$th root of unity in the complex plane, for each $n$ in $\mathbb N$, through a very elementary approach that does not depend on the complex exponential function, angles, group theory, differentiation, integration, or series. The presentation of the primitive is indirect. This approach only employs the fundamental theorem of algebra (see the elementary proofs given by de Oliveira \cite[5]{OO1} and Körner \cite{TWK}) and the four basic operations.

 A good reason to avoid the complex exponential function is justified by the fact that the theory related to it is more profound than that of the complex roots of unity (a polynomial result), see Burckel \cite{RB}. In particular, it is interesting to notice that the usual proof of the well-known Euler's Formula, $e^{i\theta}=\cos \theta + i\sin\theta$, for $\theta$ in $\mathbb R$, requires series, differentiation, and the (transcendental) numbers $e$ and $\pi$ (see Rudin ~\cite[pp. 167--169]{RU}). On the other hand, besides using the fundamental theorem of algebra, the usual group theory proofs of the existence of a primitive $n$th root of unity introduce concepts that are unnecessary to reach our objective. In addition, such existence proofs do not point out a primitive $n$th root of unity (see Artin ~\cite[pp. 49--51]{EA} and Lang ~\cite[pp. 177--178, 277--278]{SL}). 

\vspace{0,1 cm}

The study of the complex roots of unity goes back to the eighteenth century, when De Moivre (1667-1754) proved the formula $(\cos\theta + i\sin\theta)^n=\cos n\theta + i \sin n \theta $, for $\theta$ in $\mathbb R$ and $n$ in $\mathbb N$. A classical presentation of these roots, that employs Euler's Formula, can be seen in Cauchy ~\cite[pp. 196--217]{AC}. In ~\cite[p.197]{AP}, Pringsheim (1850-1941) points out how to take $n$th roots of unity, when $n=2^j$ for some $j$ in $\mathbb N$. For more historical informations, we refer the reader to Remmert ~\cite[pp. 94--96]{RR}. Nowadays, roots of unity is an important part of mathematics, especially in  number theory and discrete Fourier transform.

\section{Notations and Preliminaries}
Given $z$ in $\mathbb C$ such that $z=x+iy$, with $x$ in $\mathbb R$ and $y$ in $\mathbb R$, and $i^2=-1$, we write Re$(z)=x$ and Im$(z)=y$. The complex conjugate of $z$ is $\overline{z}= x -iy$. The fundamental theorem of algebra implies the existence of the square root function on $[0,+\infty)$. Thus, we define the absolute value of $z$ as $|z|=\sqrt{z\overline{z}}=\sqrt{x^2 + y^2}$.

\vspace{0,1 cm}

Let us fix $n$ in $\mathbb N=\{1,2,\ldots\}$. By the Fundamental Theorem of Algebra the equation $z^n=1$, with $z$ in $\mathbb C$, has $n$ solutions. Each one of these solutions is called a {\sl $n$th root of unity}. Let us denote by $w$  an arbitrary $n$th root of unity. 

\vspace{0,1 cm}

Given $w$, we have $z^n -1 = z^n -w^n = (z-w)Q(z)$, with $Q(z)=\sum_{j=0}^{n-1}z^{n-1-j}w^j$ a polynomial and $Q(w)=nw^{n-1}\neq 0$, which implies that $w$ is a simple zero of the polynomial $z^n-1$, $z$ in $\mathbb C$. Hence, there are $n$ distinct $n$th roots of unity.

\vspace{0,1 cm}

Given $k$ arbitrary in $\mathbb N$, a short computation reveals that $\overline{w}=w^{-1}$ and $w^k$ are $n$th roots of unity.  In addition, $|w^{k+1} -w^k|= |w^k(w-1)|=|w -1|$.

\vspace{0,1 cm}

If $w$ is either real or pure imaginary, then $w$ is equal to either $1$, $-1$, $i$ or $-i$.

We only assume, without proof, the fundamental theorem of algebra.

\section{A Primitive $n$th Root of Unity}

Let us consider $n$ in $\mathbb N$ and $w$, a $n$th root of unity. We say that $w$ is a {\sl primitive $n$th root of unity} if $w, w^2,\ldots,w^n=1$ are all the $n$ $n$th roots of unity.

\vspace{0,1 cm}

In order to prove the existence of a primitive $n$th root of unity we may assume that $n$ is even. This is true because if $w$ is a primitive $2n$th root of unity, then $w^2,\ldots, w^{2n}$ are all the $n$ distinct solutions of $z^n=1$. Even better, since the cases $n=2$ and $n=4$ are trivial, we may also assume $n\geq 6$.

\vspace{0,1 cm}

Given an even $n\geq 6$, the equation $z^n=1$ has a solution $w=x +iy$, with $x\neq 0$ and $y\neq 0$. Clearly $\pm w$ and $\pm \overline{w}$ are solutions of $z^n=1$ too. Therefore, there exists a $n$th root of unity with (strictly) positive real and imaginary parts.

\vspace{0,1 cm}

As a consequence, 
there exists a $n$th root of unity:
$$\left\{\begin{array}{ll}
\zeta =\zeta(n)=a+ib, \ \textrm{with}\ 0< a<1 \ \textrm{and}\ 0<b<1, \\
\textrm{satisfying}\ 0<|\zeta -1|=r,\ \textrm{with}\ r=\min\{|w -1|: \ w^n=1\ \textrm{and}\ \textrm{Im}(w)>0\}.
\end{array}
\right.$$ 
Let us notice that not only $\zeta$ satisfies $r^2=|\zeta -1|^2= (a-1)^2 +b^2=2-2a$ but $\zeta$ is also the {\sf unique} $n$th root of unity satisfying  $|\zeta-1|=r$ and $\textrm{Im}(\zeta)>0$. 
\vspace{0,1 cm}

From now on, we consider an even $n\geq 6$  and keep the notation above.

\newpage

\begin{lem} {\label L} Given an arbitrary $x$ in $[-1,1]$, we put $z=z_x=x+i\sqrt{1 - x^2}$.
\begin{itemize}
\item[(A)] $\varphi:[-a,1]\to [-1,a]$, with $\varphi(x)= \textrm{Re}(\zeta z)=a x - b\sqrt{1 -x^2}$ for each $x$ in $[-a,1]$, is bijective and strictly increasing. Its inverse $\psi:[-1,a] \to [-a,1]$ is given by $\psi(y)=\textrm{Re}(\zeta^{-1}z_y)= a y + b\sqrt{1-y^2}$, for each $y$ in $[-1,a]$. 

\item[(B)] If $x\in [-1,-a]\cup[a,1]$, then we have $z^n=1$ if and only if $x \in \{\pm 1,\pm a\}$.
\end{itemize}
\end{lem}
\begin{proof} Clearly, $-1\leq \textrm{Re}(\zeta z_x)$ and $a x-b\sqrt{1-x^2}\leq a$, for all $x$ in $[-a,1]$. We also have $\textrm{Re}(\zeta^{-1}z_y)\leq 1$ and $-a\leq ay + b\sqrt{1-y^2}$, for all $y$ in $[-1,a]$.
\begin{itemize}
\item[(A)] 
If $y$ is in $[-1,a]$, then Im$(\zeta^{-1}z_y)= a\sqrt{1 -y^2} -by$ is nonnegative on $[-1,0]$ and on $[0,a]$ (decreasing from $a$ to $0$ along $[0,a]$).
Thus, $x=\psi(y)=\textrm{Re}(\zeta^{-1}z_y)$ satisfies $z_x=\zeta^{-1}z_y$ and then
$\varphi(x)=\textrm{Re}(\zeta z_x)=y$. 

If $x$ is in $[-a,1]$, then $\textrm{Im}(\zeta z_x)=a\sqrt{1-x^2} + b x$ is nonnegative on $[-a,0]$ (increasing from $0$ to $a$) and on $[0,1]$. Thus, $y=\varphi(x)=\textrm{Re}(\zeta z_x)$ satisfies $z_y=\zeta z_x$. Hence,  $\psi(y)= \textrm{Re}(\zeta^{-1}z_y)=x$.

Evidently, $\varphi$ restricted to $[0,1]$ and $\psi$ restricted to $[-1,0]$ are increasing, with $\psi([-1,0])=[-a,b]$. Thus, $\varphi=\psi^{-1}$ restricted to $[-a,b]$ is increasing. So, the bijection $\varphi$ is strictly increasing on $[-a,1]$.

\item[(B)] If $x$ is in $(a,1)$, then we have $|z-1|^2= 2 - 2x< 2-2a =r^2$ and thus, by the definition of $r$, we get $z^n\neq 1$. If $x\in\{a,1\}$, it is obvious that $z^n=1$.

Because $n$ is even, given $x$ in $[-1,-a]$ and $z=x+i\sqrt{1 +x^2}$, it is enough to apply the last paragraph to $-x$ and $-x+i\sqrt{1-x^2}=-\overline{z}$. \qed
\end{itemize}
\end{proof}

\begin{thm} The number $\zeta$ is a primitive $n$th root of unity.
\end{thm}
\begin{proof} 
Let us define $x_k=\varphi(x_{k-1})$, with $x_0=1$ and $k\geq 1$ such that $x_{k-1}$ is in $[-a,1]$, the domain of $\varphi$. By the lemma, $\varphi$ is strictly increasing and thus $x_2=\varphi(x_1)< x_1= \varphi(x_0)=a<x_0$. We also have $x_0=\textrm{Re}(\zeta^0)$, $x_1=\textrm{Re}(\zeta)$, and $x_2=\varphi(a)=\textrm{Re}(\zeta^2)$. Hence, by induction, we obtain $x_k=\textrm{Re}(\zeta^k)< x_{k-1}$.


\vspace{0,1 cm}

Since there are $n$ $n$th roots of unity, there exists the biggest $p$ in $\mathbb N$ satisfying
$$-1\leq x_p< x_{p-1}<\cdots <x_2<x_1< x_0=1.$$
Given $k=2,\ldots,p$, the function $\varphi$ is a bijection from $[x_{k-1},x_{k-2}]$ onto $[x_k,x_{k-1}]$. Hence, by induction on $k$ and by Lemma \ref{L}(B), there are only two values of $x$ in $[x_k,x_{k-1}]$ such that $(x + i\sqrt{1-x^2})^n=1$. Namely, $x=x_k$ and $x=x_{k-1}$. 

\vspace{0,2 cm}

Let us show that $x_p=-1$. If $x_p$ is in the domain of $\varphi$, then by defining $x_{p+1}=\varphi(x_p)$ we get $x_{p+1}=\varphi(x_p)<\varphi(x_{p-1})=x_p$, against the definition of $p$. Hence, $x_p$ is in $[-1,-a)$. From Lemma \ref{L}(B), we obtain $x_p=-1$ (and $\zeta^p=-1$).

\vspace{0,2 cm}

The subintervals $[x_k,x_{k-1})$, with $k=1,\ldots,p$, form a partition of $[-1,1)$ and to each subinterval corresponds only one $n$th root of unity, in the upper hemisphere $\{z\in \mathbb C: |z|=1\ \textrm{and}\ \textrm{Im}(z)\geq 0\}$. Hence, $\zeta^0,\zeta,\ldots,\zeta^p$ are all the distinct $n$th roots of unity in the upper hemisphere and $\zeta^0,\ldots, , \zeta^p, \overline{\zeta},\ldots, \overline{\zeta^{p-1}}$
are the $n$ $n$th roots of unity. Thus, $n=2p$. Finally, we have the identities $\zeta^{p+k}\zeta^{p-k}=\zeta^{2p}=1$ and  $\zeta^{p+k}=(\zeta^{-1})^{p-k}=\overline{\zeta^{p-k}}$, for all $k=1,\ldots, p-1$.\qed

\end{proof}

\section{Remarks.}

\begin{rmk} Lemma \ref{L}(a) can be proven in a shorter, but less revealing, way through differentiation. In fact, we start by noticing that $\varphi:[-a,1]\to [-1,a]$ is continuous, satisfies $\varphi(-a)= -1$ and $\varphi(1) =a$, and
$$\varphi'(x) = a +\frac{bx}{\sqrt{1-x^2}}=\frac{a\sqrt{1 - x^2} + bx}{\sqrt{1-x^2}},\ \textrm{for all}\ x \ \textrm{in}\ (-a,1).$$
Thus, given $x$ in $[0,1)$, we obtain $\varphi'(x)\geq a>0$. If $-a<x<0$, then we have $\sqrt{1 -x^2}>\sqrt{1-a^2}=b$ and $a\sqrt{1 - x^2} + bx>ab-ab=0$. Hence, $\varphi$ is strictly increasing and, by the intermediate-value theorem, its image is $[-1,a]$.
\end{rmk}

\begin{rmk} If $n$ is prime and $w$ is a $n$th root of unity, with $w\neq 1$, then a short computation reveals that $w$ is a primitive $n$th root of unity. 
\end{rmk}

\begin{rmk} Given $m$ in $\mathbb N$ and $\zeta$, a primitive $n$th root of unity, then it is not hard to see that $\zeta^m$ is a primitive $n$th root of unity if and only if $\gcd(m,n)=1$.
\end{rmk}

\begin{rmk} Let us consider $n$ in $\mathbb N$ and $\zeta$, a primitive $n$th root of unity. Then, given $c$ in $\mathbb C$, with $c\neq 0$, and $z$, an arbitrary $n$th root of $c$, it is straightforward to see that $\zeta^0z,\zeta^1z,\ldots, \zeta^{n-1}z$ are all the $n$ distinct  $n$th roots of $c$.
\end{rmk}

\begin{rmk} Using the complex exponential function, it is clear that $e^{i\frac{2\pi}{n}}$ is a primitive $n$th root of unity. A short computation shows that $e^{i\frac{2\pi}{n}}=\zeta$.
\end{rmk}
\paragraph{Acknowledgments.} I thank professors R. B. Burckel for references ~\cite{EA} and ~\cite{AP} and his very valuable comments and suggestions. I also thank Professor Lúcio M. G. Prado for a nice discussion on this topic.

\bigskip


\noindent\textit{Oswaldo Rio Branco de Oliveira\\
Departamento de Matemática, Universidade de São Paulo\\
Rua do Matão 1010 - CEP 05508-090\\
São Paulo, SP - Brasil\\
oliveira@ime.usp.br}

\bigskip

\end{document}